\documentclass{amsart}

\usepackage{style}

\title[]{A Calabi--Yau threefold coming from two black holes}

\author{Dino Festi, Bert van Geemen}

\address{Dipartimenti di matematica ``Federigo Enriques'', Universit\`a degli studi di Milano, 20133, Milano, Italy.}
\email{dino.festi@unimi.it}
\address{Dipartimenti di matematica ``Federigo Enriques'', Universit\`a degli studi di Milano, 20133, Milano, Italy.}
\email{lambertus.vangeemen@unimi.it}

\date{\today}

\begin{document}

\maketitle

\begin{abstract}
    In this paper, we show that a set of six square roots of homogeneous polynomials in four variables, related to a binary system of black holes studied by Stefan Weinzierl, is not rationalizable.
    We prove it by showing that the variety $X$ associated to the product of four of the six square roots is not unirational.
    In particular, we show that the smooth model of $X$ is a Calabi--Yau threefold.
\end{abstract}

\section{Introduction}
Experiments in high energy physics often require predictions based on the computation of Feynman integrals, see for example~\cite{Wei22}.
Unfortunately, an exact evaluation of these integrals is often unfeasible and so practitioners resort to numerical integration.
In order to achieve results of high enough precision in a reasonable time, it is useful to ``solve'' the square root of polynomials appearing in the arguments of these integrals.
Motivated by this problem,  Marco Besier and the first author introduced the concept of \emph{rationalizability} of a square root, cf.~\cite{BF21}.
\begin{definition}
Let $f\in\IC [x_1,\ldots ,x_n]$ be a squarefree polynomial.
We say that $\sqrt{f}$ is rationalizable if there is a morphism of fields $\phi\colon \IC (x_1,\ldots,x_n)\to \IC (x_1,\ldots,x_n)$
such that $\phi(z)=z$ for every $z\in\IC$ and $\phi(f)=h^2$, for some $h\in \IC(x_1,\ldots,x_n)$.
\end{definition}
In the same article, they show that the rationalizability of a square root is equivalent to the unirationality of a certain algebraic variety associated to it, \cite[Theorem 1]{BF21}.
This equivalence is then used to give some criteria to decide rationalizability of given square roots.
The notion of rationalizability can be extended to sets of square roots of polynomials, also called \emph{alphabets}.
\begin{definition}
Let $\cA:=\{f_1,\ldots,f_r\}$ be a set of squarefree polynomials in $\IC [x_1,\dots , x_n]$.
We say that $\cA$ is rationalizable if  there is a morphism of fields $\phi\colon \IC (x_1,\ldots,x_n)\to \IC (x_1,\ldots,x_n)$
such that $\phi(z)=z$ for every $z\in\IC$ and for $i=1,\ldots,r$ we have $\phi(f_i)=h_i^2$, for some $h_i\in \IC (x_1,\ldots,x_n)$.
\end{definition}

The notion of rationalizability is further generalised to algebraic field extensions in~\cite{FH22} by Andreas Hochenegger and the first author, allowing for more results to decide rationalizability of given (sets of) polynomials.

Discussing these new developments, Stefan Weinzierl asked us about the rationalizability of an alphabet of six square roots of homogeneous polynomials in four variables related to the production of gravitational waves from a binary
system of two black holes~\cite{KRW21}.
After some straightforward changes of variables and reordering of the roots, we write Weinzierl's alphabet as the set $\cA$ of the square roots of the following six homogeneous polynomials in $R:=\IC [x,y,z,t]$:
\begin{align*}
    f_1&:=4x-z,\\
    f_2&:=4y-z,\\
    f_3&:=4x^2y-z(y-t)^2,\\
    f_4&:=4xy^2-z(x-t)^2,\\
    f_5&:=(x-y)^2-2(x+y)t+t^2,\\
    f_6&:=(x-y)^2-2(x+y)(z+t)+(z+t)^2.
\end{align*}
Unfortunately, this set seems too complicated for the methods presented in the previous papers, and requires an ad-hoc study that we present in this paper.
As by-product, we also developed a strategy to \emph{guess} (but not prove/disprove) the rationalizability of square roots in certain cases.

The rationalizability of $\cA$ would imply the rationalizability of the square root of every product of the polynomials $f_i$.
Conversely, if one can prove that the square root of one of the products is not rationalizable, then $\cA$ is not rationalizable.
Weinzierl has studied several products of two and three polynomials (see~\autoref{p:Weinzierl}), 
proving the rationalizability of all of them but leaving the rationalizability of the whole alphabet open.
In these notes we are going to show that the set is not rationalizable.

\begin{theorem}\label{t:Main}
The square root $\sqrt{f_1f_2f_3f_4}$ is not rationalizable and hence neither is the set $\cA$.
\end{theorem}

The proof of \autoref{t:Main} is presented in~\autoref{s:Resolution}.
In~\autoref{s:Guessing} we propose a method to guess the rationalizability of square roots in certain cases.
This method provided a decisive hint towards the non-rationalizability of $\sqrt{f_1f_2f_3f_4}$, see~\autoref{ss:GuessingWeinzierlOctic}.

\section*{Acknowledgements}
We thank Stefan Weinzierl for posing the original question. 
We also thank S{\l{}}awomir Cynk and Bartosz Naskr\c{e}cki for the insight into the modularity of the threefold $X$.
The first author is also grateful to Andreas Hochenegger and Matthias Sch\"{u}tt for useful conversations about this work.

\section{The double octic}\label{s:DoubleOctic}

As stated above, Stefan Weinzierl studied the rationalizability of several products of two and three polynomials, obtaining the result below.

\begin{proposition}[Weinzierl, private communication]\label{p:Weinzierl}
The square roots of the following products are rationalizable:
\[
\begin{array}{ccc}
  f_1f_5,   & f_2f_5, &f_1f_2f_3, \\
   f_1f_2f_4,  & f_1f_2f_5, & f_1f_4f_3,\\
   f_1f_3f_5,  & f_2f_3f_5, &f_2f_4f_5.
\end{array}
\]
\end{proposition}
The proof of the proposition above is made easier by the relations between the polynomials, stated below.
\begin{lemma}
The involution of $R$ defined by switching $x$ and $y$,
exchanges $f_1$ with $f_2$ and $f_3$ with $f_4$;
it fixes $f_5$ and $f_6$.

The automorphism of $R$ defined by  $t\mapsto z+t$ sends $f_5$ to $f_6$ and fixes $f_1$ and $f_2$.
\end{lemma}
\begin{proof}
Immediate by direct computations.
\end{proof}

The next natural step in the study of $\cA$ is then to consider the product of four polynomials.

\subsection{A product of four polynomials}
We define the polynomial $f$ of degree $8$ as the product
\begin{equation}
    f:=f_1f_2f_3f_4.
\end{equation}
Let $\IP$ denote the weighted projective space $\IP(1,1,1,1,4)$ with coordinates $x,y,z,t,w$ of weights $1,1,1,1,4$, respectively.
We define  the variety $\cX\subset \IP$ as
\begin{equation}
    \cX\colon\; w^2=f\; .
\end{equation}
Then the rationalizabilty of $\sqrt{f}$ is equivalent to the unirationality of $\cX$ (cf.~\cite[Theorem 1.1]{BF21}).

Notice that $\cX$ is a threefold given as a double cover of $\IP^3$ branched along the octic surface $\cV\subset \IP^3$ defined by
\begin{equation}
    \cV\colon\; f=0\; .
\end{equation}
If $\cV$ were smooth, then $\cX$ would also be smooth.
Moreover, a smooth double cover of $\IP^3$ branched above an octic surface (a \emph{double octic threefold} or \emph{octic double solid}) is a Calabi--Yau threefold (cf.~\autoref{d:CY}), hence it is not unirational.
This would already answer the question about the rationalizabilty of $\cA$.
Unfortunately, $\cV$ is (very) singular,
making hard to see whether there is a desingularization of $\cX$ that is a Calabi--Yau or not.
This is why we look for a more convenient birationally equivalent model.

\subsection{A more convenient model}
Consider the Cremona-like transformation of $\IP^3$ given by
\begin{equation*}
    \sigma\colon \; \IP^3\dashrightarrow \IP^3\, , \quad (x:y:z:y)\mapsto(xz:yz:4xy:tz).
\end{equation*}
Let $X\subset \IP$ be the threefold defined by
\begin{equation}\label{eq:DoubleOctic}
    X\colon \; w^2=xy(x-z)(y-z)(yz-(x-t)^2)(xz-(y-t)^2)\; .
\end{equation}

Notice that $X$ is again a double cover of $\IP^3$ ramified above an octic surface, namely the surface
\[
B\colon \; xy(x-z)(y-z)(yz-(x-t)^2)(xz-(y-t)^2)=0 \; \subset \IP^3\, .
\]
\begin{lemma}\label{l:sigma}
The following statements hold.
\begin{enumerate}
    \item $\sigma$ is a rational involution; in particular it is a birational map.
    \item $\sigma$ induces a birational morphism between $X$ and $\cX$.
\end{enumerate}
\end{lemma}
\begin{proof}
It easy to see that $\sigma\circ\sigma$ is the identity, proving the first statement.

The map $\sigma$ sends the surface $\cV$ to the surface $B'\subset \IP^3$ defined by
$$
B'\colon \; x^3y^3z^4(x-z)(y-z)(yz-(x-t)^2)(xz-(y-t)^2)=0.
$$
Let $X'\subset \IP(1,1,1,1,8)$ be the double cover of $\IP^3$ branched above $B'$ and let $w'$ be the coordinate of weight $8$. 
As $B'$ is birationally equivalent to $\cV$ (via $\sigma$), we have that $X'$ is birationally equivalent to $\cX$;
on the other hand, the change of variable $w'=xyz^2w$ shows that $X$
 and $X'$ are also birationally equivalent,
 concluding the proof of the second statement.
 \end{proof}
\begin{corollary}
$\cX$ is unirational if and only if $X$ is.
\end{corollary}
\begin{proof}
By~\autoref{l:sigma}, $X$ and $\cX$ are birationally equivalent. 
As unirationality is preserved by birational morphisms, 
the statement follows.
\end{proof}

As $X$ is a double cover of $\IP^3$,
its singularities can only come from the singularities of the branch locus $B$.

\begin{lemma}\label{l:BranchComponents}
The surface $B$ has six irreducible components, named $B_1,\dots,B_6$, defined as follows:
\begin{align*}
    B_1&:x=0\, ,\\
    B_2&:y=0\, ,\\
    B_3&:x-z=0\, ,\\
    B_4&:y-z=0\, ,\\
    B_5&:yz-(x-t)^2=0\, ,\\
    B_6&:xz-(y-t)^2=0\, .
\end{align*}
The components $B_1,B_2,B_3,B_4$ are planes.
The components $B_5, B_6$ are singular quadrics, with a single node in $(1:0:0:1)$ and $(0:1:0:1)$, respectively.
\end{lemma}
\begin{proof}
By direct computations. 
\end{proof}

\subsection{The singular locus}
In this section we study the singular locus of the branch locus $B$.
From~\autoref{l:BranchComponents} we know that $B$ has six irreducible components: four planes and two quadrics; each quadric has a single node.
The singular locus of $B$ is given by the intersection of the irreducible components  (the two nodes  lying in these intersections).

\begin{definition}
For $i,j\in\{1,\dots ,6\},\, i<j$, we define $B_{i,j}:=B_i\cap B_j$.
\end{definition}
\begin{lemma}\label{l:SingularComponents}
The singular locus of $B$ is the union of 18 rational curves.
It is non-reduced.
More specifically, the following statements hold:
\begin{itemize}
    \item the intersections $B_{1,2}, B_{1,3}, B_{1,4}, B_{2,3}, B_{2,4}, B_{3,4}$ are lines;
    \item the intersections $B_{1,6}$, resp. $B_{2,5}$, is a double line, that is, $B_1$ and $B_6$, resp. $B_2$ and $B_5$, are tangent in their intersection;
    \item $B_{1,5}, B_{2,6}, B_{3,5}, B_{4,6}$ are smooth conics;
    \item $B_{3,6}$ and $B_{4,5}$ are the union of two incident lines, we write
    \begin{itemize}
        \item $B_{3,6}=B_{3,6}^1 \cup B_{3,6}^2$, with the two lines meeting in $(0:1:0:1)$,
        \item $B_{4,5}=B_{4,5}^1 \cup B_{4,5}^2$, with the two lines meeting in $(1:0:0:1)$;
    \end{itemize}
    \item $B_{5,6}$ is the union of two smooth conics meeting in $(1:1:0:1),(1:1:4:3)$, we write $B_{5,6}=B_{5,6}^1 \cup B_{5,6}^2$.
\end{itemize}
\end{lemma}
\begin{proof}
By direct computations.
\end{proof}
\begin{remark}
Notice that the equations defining the irreducible components of $B_{3,6}, B_{4,5},$ and $B_{5,6}$ can be retrieved from~\autoref{t:SingularPoints}.
\end{remark}
\begin{definition}
Let $\cB$ denote the set of the irreducible components of the singular locus of $B$, that is, $\cB$ is the set of the 18 rational curves defined in~\autoref{l:SingularComponents}.
\end{definition}
\begin{lemma}
The curves in $\cB$ intersect in 16 points.

In~\autoref{t:SingularPoints}, for each intersection point we indicate its multiplicity in $B$, the components on which it lies, and the singular curves on which it lies.
\end{lemma}
\begin{proof}
By direct computations.
\end{proof}
\begin{table}[ht]
\begin{tabular}{|l|l|l|l|}
\hline
$(1:0:0:1)$ & $4$ & $B_2, B_4, B_5$ & $B_{2,4}, B_{2,5}, B_{4,5}^1, B_{4,5}^2 $ \\
$(0:1:0:1)$ & $4$ & $B_1, B_3, B_6$ & $B_{1,3}, B_{1,6}, B_{3,6}^1, B_{3,6}^2$ \\
$(0:-1:-1:1)$ & $3$ & $B_1, B_4, B_5$ & $B_{1,4}, B_{1,5}, B_{4,5}^2$ \\
$(-1:0:-1:1)$ & $3$ & $B_2, B_3, B_6$ & $B_{2,3}, B_{2,6}, B_{3,6}^2$ \\
$(1:0:0:0)$ & $3$ & $B_2, B_4, B_6$ & $B_{2,4}, B_{2,6}, B_{4,6}$ \\
$(0:1:0:0)$ & $3$ &$B_1, B_3, B_5$  & $B_{1,3}, B_{1,5}, B_{3,5}$ \\
$(0:0:1:0)$ & $4$ & $B_1, B_2, B_5, B_6$ & $B_{1,2}, B_{1,5}, B_{1,6}, B_{2,5},  B_{2,6}, B_{5,6}^2$ \\
$(0:0:0:1)$ & $4$ & $B_1, B_2, B_3, B_4$ & $B_{1,2}, B_{1,3}, B_{1,4}, B_{2,3},  B_{2,4}, B_{3,4}$ \\
$(0:1:1:1)$ & $4$ & $B_1, B_4, B_5, B_6$ & $B_{1,4},  B_{1,5}, B_{1,6}, B_{4,5}^1, B_{4,6}, B_{5,6}^1 $ \\
$(1:0:1:1)$ & $4$ & $B_2, B_3, B_5, B_6$ & $B_{2,3}, B_{2,5}, B_{2,6}, B_{3,5}, B_{3,6}^1, B_{5,6}^1$ \\
$(1:1:0:1)$ & $2$ & $B_5, B_6$ & $B_{5,6}^1, B_{5,6}^2$  \\
$(1:1:1:0)$ & $4$ & $B_3, B_4, B_5, B_6$ & $B_{3,4}, B_{3,5}, B_{3,6}^2,  B_{4,5}^2, B_{4,6}, B_{5,6}^2 $ \\
$(1:1:1:2)$ & $4$ & $B_3, B_4, B_5, B_6$ & $ B_{3,4}, B_{3,5}, B_{3,6}^1, B_{4,5}^1, B_{4,6}, B_{5,6}^2$ \\
$(1:1:4:3)$ & $2$ & $B_5, B_6$ & $B_{5,6}^1, B_{5,6}^2$  \\
$(1:4:1:3)$ & $3$ & $B_3,B_5,B_6$ & $B_{3,5}, B_{3,6}^2, B_{5,6}^1$ \\
$(4:1:1:3)$ & $3$ & $B_4, B_5, B_6$ & $B_{4,5}^2, B_{4,6}, B_{5,6}^1$ \\
\hline
\end{tabular}
\caption{For each point, we indicate its multiplicity on $B$, the irreducible components of $B$ on which it lies, and the singular curves on which it lies.}
\label{t:SingularPoints}
\end{table}

\section{Resolution of the singularities}\label{s:Resolution}
In this section we show that by suitably resolving the singularities of $X$ we obtain a Calabi--Yau threefold, hence $X$ is not unirational.
We follow the procedure to desingularize $X$ described in~\cite[\S 2]{CS98} and, later, in~\cite[\S 4.1]{Mey05}.
Some concrete applications can be found in~\cite[Example 3]{Cyn02} and~\cite[Example 5.2]{CvS06}.

\begin{definition}\label{d:arrangement}
Let $U$ be a smooth threefold and 
 $S\subset U$  a surface.
We call $S$ \emph{an arrangement} if locally it looks like an union of planes,
that is, $S$ is the union of smooth irreducible surfaces  $S_1,...,S_r$  such that:
\begin{enumerate}
    \item for any $i\neq j$, the surfaces $S_i, S_j$ either intersect transversally along a smooth curve $C_{i,j}$ or are disjoint;
    \item the curves $C_{i,j},\, C_{k,l}$ either intersect transversally or coincide.  
\end{enumerate}

If $U=\IP^3$ and the surfaces $S_1,\dots ,S_r$ are of degree $d_1,\ldots , d_r$ such that $d_1+\cdots + d_r=8$,
then $S$ is called an \emph{octic} arrangement.

An irreducible curve $C\subset S$ is said to be $q$-\emph{fold} if it lies on exactly $q$ of the surfaces $S_1,\dots , S_r$;
a point $P\in S$ is said to be $p$-\emph{fold} if it lies on exactly $p$ of the surfaces $S_1,\dots , S_r$.
\end{definition}

\begin{remark}\label{r:NonArrangement}
One sees immediately that $B$ is not an octic arrangement:
the two quadric components $B_5$ and $B_6$ are not smooth (cf.~\autoref{l:BranchComponents});
moreover, they are tangent to $B_2$ and $B_1$, respectively, and
the intersections $B_{3,6}, B_{4,5}$ and $B_{5,6}$ are reducible.
\end{remark}

As noted in~\autoref{r:NonArrangement}, the branch locus $B$ of $X\subset\IP^3$ is not an octic arrangement.
We want to show that by blowing up $\IP^3$ in some of the curves in $\cB$ we obtain 
threefold $U^{(5)}$ such that the strict transform $B^{(5)}\subset U^{(5)}$ of $B^{(4)}$ is an arrangement.
Then we proceed as in~\cite{CS98} and~\cite[Chapter 4]{Mey05}.

\subsection{The first blow up}
In this subsection we illustrate the first step of the resolution process,
 the blow up of $\IP^3$ along the line $B_{1,6}$.

\begin{definition}
We define 
\[
\sigma^{(1)}\colon \; U^{(1)}:=\Bl_{1,6}\IP^3:=\Bl_{B_{1,6}}\IP^3\to \IP^3=:U^{(0)}
\]
to be the blow up of $\IP^3=:U^{(0)}$ along the line $B_{1,6}$.

We denote by $B^{(1)}\subset U^{(1)}$ the strict transform of $B\subset \IP^3$
and by $B_i^{(1)}\subset B^{(1)}$ the strict transform of its components.
Also, let $K:=K_{\IP^3}$ and $K^{(1)}$  denote the canonical divisors of $\IP^3$ and $U^{(1)}$, respectively.
\end{definition}

\begin{lemma}\label{l:B1}
The divisor $B^{(1)}\subset U^{(1)}$ is linearly equivalent to $-2K^{(1)}$ and it inherits all the singularities of $B$, cf.~\autoref{l:SingularComponents}, except:
\begin{itemize}
    \item $B^{(1)}_6$ is smooth,
    \item $B_1^{(1)}$ and $B_6^{(1)}$ intersect transversally in a line,
    \item $B_3^{(1)}$ and $B_6^{(1)}$ intersect in two disjoint lines.
\end{itemize}
\end{lemma}
\begin{proof}
The branch surface $B_6$ is a singular quadric with a node in $P:=(0:1:0:1)$ which is tangent to the branch plane $B_1$ in the line $B_{1,6}\colon x=y-t=0$.
The node $P$ also lies on the branch surfaces $B_1$ (and hence on the tangency line $B_{1,6}$) and $B_3$.
For convenience, we introduce the variable $u:=y-t$, so that we can write
\[
B_1\colon x=0,\;\; B_3\colon x-z=0,\;\; B_6\colon xz=u^2,\;\; B_{1,6}\colon x=u=0,\;\; P=(0:0:0:1) 
\]
in $\IP^3(x,u,z,t)$.

Identify the subset $\{t\neq 0\}\subset \IP^3$ with $\IA^3(x,u,z)$ and 
consider the blow up $\sigma\colon\Tilde{\IA}^3\to \IA^3$ of $\IA^3$ in $\IA^3\cap B_{1,6}$:
$$
\Tilde{\IA}^3:=\{ ((x,u,z),(v_0:v_1)) \; :\; v_0x-v_1u \}\subset \IA^3\times \IP^1.
$$
The exceptional divisor is then 
$$
E:=\{((0:0:z),(v_0:v_1)) \}\cong \IA^1\times \IP^1.
$$
We denote the standard affine open subsets of $\Tilde{\IA}^3$ by
\begin{align*}
    \Tilde{\IA}^3_0:=\{v_1\neq 0\}&=\IA^3(x,z,v_0),\; (u=v_0x);\\
    \Tilde{\IA}^3_1:=\{v_0\neq 0\}&=\IA^3(u,z,v_1),\; (x=v_1u).
\end{align*}
Notice that
\begin{align*}
    E_0&:= E\cap \Tilde{\IA}^3_0 = \{ x=0\},\\
    E_1&:= E\cap \Tilde{\IA}^3_1 = \{ u=0\}.
\end{align*}
The inverse image of $B_1\colon x=0$ in $\Tilde{\IA}^3$ lies in $\{v_1u=0\}$;
since $u$ is not identically zero on $B_1$, the strict transform ${B}_1^{(1)}$ of $B_1$ lies in $v_1=0$.
Hence, ${B}_1^{(1)}\cap \Tilde{\IA}^3_0=\emptyset$.
Therefore, to study ${B}_1^{(1)}$ we only need to consider $\Tilde{\IA}^3_1$.
The preimages of $B_3$ and $B_6$ are given by
\begin{align*}
    \sigma^{-1}(B_3)&\colon \; v_1u=z,\\
    \sigma^{-1}(B_6)&\colon \; u(v_1z-u)=0;
\end{align*}
denote by $B^{(1)}_i$ the strict transform of $B_i$ in $\Tilde{\IA}^3_1$.
As $E_1=\{u=0\}$ we have 
\begin{equation}\label{eq:PullBack16}
    \sigma^*(B_1+B_2+B_3)\cap \Tilde{\IA}^3_1=2E_1+B^{(1)}_1+B^{(1)}_3+B^{(1)}_6\, .
\end{equation}

By construction,  $\Tilde{\IA}^3\subset U^{(1)}$ and the projection $\sigma^{(1)}\colon U^{(1)}\to \IP^3$ restricted to $\Tilde{\IA}^3$ is $\sigma$.
From~\eqref{eq:PullBack16} it follows that $E$ is not in the branch locus of $X^{(1)}$.

Then, to compute the canonical divisor $K_{X^{(1)}}$ of $X^{(1)}$,
notice that the branch locus   of $X$ on $\IP^3$ is 
$B=B_1+\cdots +B_6$, a divisor of degree 8 and hence
\begin{equation}\label{eq:BK}
    B=-2K_{\IP^3}.
\end{equation}
Its strict transform is 
\begin{equation}\label{eq:tr1B}
    B^{(1)}={\sigma^{(1)}}^*B-2E.
\end{equation}
By adjunction
\begin{equation}\label{eq:tr1K}
    {\sigma^{(1)}}^*K_{\IP^3}=K^{(1)}-E.
\end{equation}
Equations~\eqref{eq:BK}, \eqref{eq:tr1B}, and~\eqref{eq:tr1K} yield
\begin{equation}\label{eq:B1}
    B^{(1)}=-2K^{(1)},
\end{equation}
proving the first statement.

The second statement follows immediately from the definition of blow up,
noticing that the node $P$ of $B_1$ lies on $B_{1,6}$,
and this point is also the point of intersection of the two components of $B_{3,6}$.
\end{proof}

\subsection{The second blow up}
We proceed by blowing the strict transform of $B_{2,5}$.
\begin{definition}
We define
$$
\sigma^{(2)}\colon \; U^{(2)}\to U^{(1)}
$$
as the blow up of $U^{(1)}$ along $B^{(1)}_{2,5}$, the strict transform of $B_{2,5}$ in $U^{(1)}$.

We denote by $B^{(2)}$ and  $B_i^{(2)}$  the strict transforms in $U^{(2)}$ of $B^{(1)}$
and  $B_i^{(1)}$, respectively.
We also denote by $K^{(2)}$ the canonical divisor of $U^{(2)}$.
\end{definition}
We obtain a result analogous to~\autoref{l:B1}.

\begin{lemma}\label{l:B2}
The divisor  $B^{(2)}$ is linearly equivalent to $-2K^{(2)}$ and it inherits all the singularities of $B^{(1)}$, cf.~\autoref{l:B1}, except:
\begin{itemize}
    \item $B^{(2)}_6$ is smooth,
    \item $B_2^{(2)}$ and $B_5^{(2)}$ intersect transversally in a line,
    \item $B_4^{(2)}$ and $B_5^{(2)}$ intersect in two disjoint lines.
\end{itemize}
\end{lemma}
\begin{proof}
Let $E$ be the exceptional divisor of $U^{(2)}$ introduced by the blow up $\sigma^{(2)}$.
Then, reasoning as in the proof of~\autoref{l:B1}, 
one can show that $B^{(2)}={\sigma^{(2)}}^*B^{(1)}-2E$ and ${\sigma^{(2)}}^*K^{(1)}=K^{(2)}-E$.
As above, these two equalities and~\eqref{eq:B1}
 yield 
$$
B^{(2)}=-2K^{(2)}.
$$

Again, the second statement follows immediately form the definition of blow up, noticing that the node of $B_5^{(1)}$ lies on $B_{2,5}^{(1)}$, 
and this node is also the intersection point of the two components of $B_{4,5}^{(1)}$.
\end{proof}

\subsection{The third blow up}
We proceed as above, blowing up the strict transform of one of the components of $B_{5,6}$. Recall that $B_{5,6}$ has two irreducible components meeting in two points. 
The strict transform $B_{5,6}^{(2)}\subset U^{(2)}$ stays reducible, with the two components still meeting in two points.
This feature prevents $B^{(2)}$ from being an arrangement.
In order to solve this obstruction, we blow up one of the two components.
\begin{definition}
We define
$$
\sigma^{(3)}\colon \; U^{(3)}\to U^{(2)}
$$
as the blow up of $U^{(2)}$ along $(B^{1}_{5,6})^{(2)}$, the strict transform of $(B^{1}_{5,6})^{(1)}$ in $U^{(2)}$.

We denote by $B^{(3)}$ and  $B_i^{(3)}$  the strict transforms in $U^{(3)}$ of $B^{(2)}$
and  $B_i^{(2)}$, respectively.
We also denote by $K^{(3)}$ the canonical divisor of $U^{(3)}$.
\end{definition}
\begin{lemma}
The divisor  $B^{(3)}$ is linearly equivalent to $-2K^{(3)}$ and it inherits all the singularities of $B^{(2)}$, cf.~\autoref{l:B2}, except that $B_{5}^{(3)}$ and $B_6^{(3)}$ intersect in a single smooth conic.
\end{lemma}
\begin{proof}
As for~\autoref{l:B2}.
\end{proof}

\subsection{The fourth and fifth blow ups}
We proceed our plan of resolving the singularities of $B$ in order to obtain an arrangement.
In order to do so, we have to make sure that the strict transform of $B_{3}$ and $B_6$, respectively $B_4$ and $B_5$, meet in a smooth irreducible curve.
As $B_3\cap B_6$, resp. $B_4\cap B_5$, meet in two incident lines, and so do their strict transforms in $U^{(3)}$, it is enough to blow up one component of each intersection.
Therefore,
we define the blow up 
$$
\sigma^{(4)}\colon \; U^{(4)}\to U^{(3)}
$$
  of $U^{(3)}$ along $(B^{1}_{3,6})^{(3)}$.
Subsequently, 
we define 
$$
\sigma^{(5)}\colon \; U^{(5)}\to U^{(4)}
$$
as the blow up of $U^{(4)}$ along $(B^{1}_{4,5})^{(4)}$.
For $j=4,5$, we  define 
$B^{(j)}$ and  $B_i^{(j)}$  as the strict transforms in $U^{(j)}$ of $B^{(j-1)}$
and  $B_i^{(j)}$, respectively.
Moreover, for $j=4,5$ we denote by $K^{(j)}$ the canonical divisor of $U^{(j)}$.

\begin{lemma}
The divisor  $B^{(5)}$ is linearly equivalent to $-2K^{(5)}$ and it is an arrangement in the sense of~\autoref{d:arrangement};
  it only has 2-fold curves and at most 4-fold points.
\end{lemma}
\begin{proof}
  The first statement is proven as in the previous cases.
  
  To see that $B^{(5)}$ is an arrangement, 
  note that $B^{(5)}$ is isomorphic to $B=B^{(0)}$ outside the curves that have been blown up: $B_{1,6}, B_{2,5},B_{5,6}^1, B_{3,6}^1, B_{4,5}^1$.
  As noted in~\autoref{r:NonArrangement}, these are exactly the loci not meeting the conditions for $B$ to be an arrangement.
  Indeed, $B^{(5)}$ is the union of six smooth surfaces:
  \begin{itemize}
      \item $B^{(5)}_1,\ldots ,B^{(5)}_4$ are strict transforms of planes;
      \item $B^{(5)}_5$ and $B^{(5)}_6$ are the strict transforms of quadrics with a node blown up (also) along a line containing the node, hence they are smooth;
      \item for every $i,j\in\{1,\ldots ,6\}, i\neq j$, the intersection $B^{(5)}_i\cap B^{(5)}_j$ is a smooth curve: either a line or a conic.
  \end{itemize}
  
  As $B$ only has 2-fold curves and at most 4-fold points, 
  so does $B^{(5)}$. 
  Notice that some of the 4-fold points have not been affected by the blow ups, e.g., $(0:0:0:1)$, hence $B^{(5)}$ does have 4-fold points.
\end{proof}

\subsection{Further blow ups}
Following~\cite{CS98} and~\cite[\S 4.1]{Mey05}, we proceed by blowing up the 4-fold points and then the double curves of $B^{(5)}\subset U^{(5)}$,
eventually obtaining a smooth surface $\Tilde{B}:=B^{(n)}\subset U^{(n)}=:U$.
Let $\Tilde{X}$ be the double cover of $U$ ramified above $\Tilde{B}$.

For the convenience of the reader, we also recall the definition of a Calabi--Yau threefold:
a smooth threefold $Y$ is a Calabi--Yau threefold if and only if its canonical divisor $K_Y$ is trivial and $H^i(Y,\cO_Y)=0$ for $i=1,2$;
see also~\autoref{d:CY}.

\begin{proposition}\label{p:XCY}
 The threefold $\Tilde{X}$ is smooth and  birationally equivalent to $X$.
 Moreover, it is a Calabi--Yau threefold.
\end{proposition}
\begin{proof}
Let $\sigma=\sigma^{(n)}\circ\ldots\circ \sigma^{(1)}$ be the composition of all the blow ups
$$
\sigma\colon \; U=U^{(n)}\to U^{(n-1)} \cdots U^{(1)}\to U^{(0)}=\IP^3\; .
$$
Then $\sigma$ is a birational map between $U$ and $\IP^3$, 
inducing a birational map between $\Tilde{B}$ and $B$.
Therefore $\Tilde{X}$ and $X$ are birationally equivalent.

It remains to prove that $\Tilde{X}$ is a Calabi--Yau threefold, i.e.\,
it is smooth, has trivial canonical divisor, and $H^i(\Tilde{X},\cO_{\Tilde{X}})=0$ for $i=1,2$.

As $\Tilde{X}$ is the double cover of $U$, which is smooth,  branched above the smooth (reducible) surface $\Tilde{B}$, we have that $\Tilde{X}$ is also smooth.

After every blow up we have that $B^{(k)}=-2K^{(k)}$, where $K^{(k)}$ is the canonical divisor of $U^{(k)}$. 
Hence 
$$
\Tilde{B}=B^{(n)}=-2K^{(n)}=-2K_{U},
$$ 
yielding $K_{\Tilde{X}}=0$.

Let $\pi\colon \Tilde{X}\to U$ denote the double cover projection.
As $\pi$ is finite, it is affine and hence $H^i(\Tilde{X},\cO_{\Tilde{X}})\cong H^i(U,\pi_*\cO_{\Tilde{X}})$ for each $i\geq 0$.
As $\Tilde{X}$ is the double cover of $U$ branched above $\Tilde{B}=-2K_U$,
it follows that 
$$
\pi_* \cO_{\Tilde{X}} = \cO_U\oplus \cO_U (K_U ),
$$
and hence 
$$
H^i(X,\cO_{\Tilde{X}})\cong H^i(U,\cO_{\Tilde{X}})\oplus H^i(U,\cO_{\Tilde{X}} (K_U)).
$$
By Serre duality we have that
$$
H^i(U,\cO_U (K_U))\cong H^{3-i}(U,\cO_U).
$$
So we are only left to compute $H^{i}(U,\cO_U)$ for $i=1,2$.

To see this, consider the sequence of blow ups
$$
U:=U^{(n)}\to U^{(n-1)}\to \cdots U^{(1)}\to U^{(0)}:=\IP^3,
$$
where $\sigma^{(\ell)}\colon U^{(\ell)}\to U^{(\ell-1)}$ is the blow up of $U^{(\ell-1)}$ along a curve or a point, say $Z^{(\ell-1)}$ or  $P^{(\ell-1)}$, respectively.

If $\sigma^{(\ell)}$ is the blow up along the curve $Z^{(\ell-1)}$, 
then $\codim_{U^{(\ell-1)}} Z^{(\ell-1)}=2$.
From~\cite[Theorem 7.31]{Voi07}, it follows that
$$
H^i(U^{(\ell)},\IZ)\cong H^{i}(U^{(\ell-1)},\IZ)\oplus H^{i-2}(Z^{(\ell-1)},\IZ).
$$
Then, for $i=1$, we have 
\[
H^1(U^{(\ell)},\IZ)\cong H^1(U^{(\ell-1)},\IZ)\, ;
\]
for $i=2$, 
\[
H^2(U^{(\ell)},\IZ)\cong H^{2}(U^{(\ell-1)},\IZ)\oplus H^{0}(Z^{(\ell-1)},\IZ)\cong
H^{2}(U^{(\ell-1)},\IZ)\oplus\IZ,
\]
as $Z^{\ell-1}$ is a curve
and hence $H^{0}(U^{(\ell-1)},\IZ)\cong \IZ$.
Moreover, $H^{0}(U^{(\ell-1)},\IZ)$ is of type $(0,0)$, and hence its contribution to $H^2(U^{\ell},\IZ)$ is of type $(1,1)$.

The above reasoning applies also if $\sigma^{(\ell)}$ is the blow up along the point $P^{(\ell)}$.
In this case $\codim_{U^{(\ell)}} P^{(\ell)}=3$.
Then, from~\cite[Theorem 7.31]{Voi07}, it follows that
\[
H^i(U^{(\ell)},\IZ)\cong H^{i}(U^{(\ell-1)},\IZ)\oplus H^{i-2}(P^{(\ell-1)},\IZ)\oplus H^{i-4}(P^{(\ell-1)},\IZ).
\]
From this it follows that
\begin{align*}
    H^1(U^{(\ell)},\IZ)&\cong H^1(U^{(\ell-1)},\IZ)\, ,\\
    H^2(U^{(\ell)},\IZ)&\cong H^{2}(U^{(\ell-1)},\IZ)\oplus \IZ\, ,
\end{align*}
exactly as in the previous case.
Also in this case, the contribution to $H^2(U^{(\ell)},\IZ)$ comes from $H^0(P^{(\ell-1)},\IZ)\cong \IZ$ and it is of type $(1,1)$.

For $\ell=1$ we then have
\begin{align*}
    H^1(U^{(1)},\IZ)&\cong H^1(\IP^3,\IZ)=0\, ,\\
    H^2(U^{(1)},\IZ)&\cong H^2(\IP^3,\IZ)\oplus H^0(Z^{(1)},\IZ)\cong \IZ^2\, .
\end{align*}
As $H^2(\IP^3,\IZ)$ is of type $(1,1)$, so is $H^2(U^{(1)},\IZ)$.
Then, proceeding by induction, we obtain
\begin{align*}
    H^1(U,\IZ)\cong H^1(\IP^3,\IZ)&=0\, ,\\
    H^2(U,\IZ)=H^2(U^{(n)},\IZ)&\cong \IZ^{n+1}\, ,
\end{align*}
with $H^2(U,\IZ)$ purely of type $(1,1)$.
From $H^1(U,\IZ)=0$ it follows that $0=h^1(U)=h^{1,0}+h^{0,1}$ and hence 
$$
0=h^{0,1}=\dim H^1(U,\cO_U),
$$
concluding that $H^1(U,\cO_U)=0$;
from $H^2(U,\IZ)\cong \IZ^{n+1}$, of type $(1,1)$, it follows that
$$
n+1=h^2(U)=h^{2,0}(U)+h^{1,1}(U)+h^{0,2}=h^{2,0}(U)+(n+1)+h^{0,2},
$$
yielding $h^{2,0}=h^{0,2}=0$.
In particular, $h^{0,2}=0$ means that $H^2(U,\cO_U)=0$,
concluding the proof.

\end{proof}

We are now ready to prove the main theorem.

\begin{proof}[Proof of~\autoref{t:Main}]
From~\cite[Proposition 18]{BF21}, the square root $\sqrt{f_1f_2f_3f_4}$ is rationalizable if and only if $Y$ is unirational.
As $Y$ is birationally equivalent to $X$, then the condition above is equivalent to the unirationality of $X$.
From~\autoref{p:XCY} we know that $X$ is birationally equivalent to $\Tilde{X}$
and that $\Tilde{X}$ is a Calabi--Yau threefold, hence it is not unirational.
Therefore $\sqrt{f_1f_2f_3f_4}$ is non-rationalizable.

The non-rationalizability of $\cA$ follows 
 immediately from the first statement and~\cite[Proposition 47]{BF21}.
\end{proof}

\begin{remark}
Notice that in order to prove the theorem, it would have been enough to show that $\Tilde{X}$ has a trivial, and in fact even just effective, canonical divisor. 
Indeed, this would already imply that $h^{3,0}=1\gneq 0$ and hence $\Tilde{X}$ cannot be unirational.  
\end{remark}

\section{Guessing unirationality via point counting}\label{s:Guessing}

In this section we show how to use modularity and point counting to \emph{guess} the unirationality of a given threefold $Y$ defined over $\IQ$, cf.~\autoref{c:NotUnirat}.
The technique reported below is very quick to employ but subject to assumptions that are often hard to check, 
making this method not viable to give a proof.
A consistency check can be made under the assumption that the smooth model $\Tilde{Y}$ of $Y$ is a K3 surface or a Calabi--Yau threefold, cf.~\autoref{c:SingK3} and~\autoref{c:Bert}, respectively.

\subsection{Introduction to modularity of Calabi--Yau varieties}
In what follows, $Y$ will denote a (not necessarily smooth) complex projective variety of dimension $d$ that can be defined over $\IQ$.
In this subsection we will mostly follow~\cite[Chapter 1]{Mey05}.
\begin{definition}\label{d:CY}
We say that $Y$ is a \emph{Calabi--Yau} variety if it is smooth, $H^i(Y,\cO_Y)=0$ for every $0<i<d$, and the canonical divisor $K_Y\sim 0$ is trivial.
\end{definition}
From the second condition and by Serre duality we get that $$
h^{0,0}(Y)=h^{d,0}(Y)=h^{0,d}(Y)=h^{d,d}(Y)=1,
$$
where $h^{p,q}(Y):=\dim_\IC H^q(Y,\Omega^p_Y)$ is the $p,q$-th Hodge number of $Y$.

\begin{example}
A Calabi--Yau variety of dimension $d=1$ (equipped with a rational point) is an \emph{elliptic curve}.
It has the following Hodge diamond.
\[
\begin{array}{ccc}
     & 1 & \\
    1 & ~ & 1\\
      & 1 &
\end{array}
\]
\end{example}

\begin{example}
A Calabi--Yau variety of dimension $d=2$  is a \emph{K3 surface}.
It has the following Hodge diamond.
\[
\begin{array}{ccccc}
   &  & 1 & &\\
   & 0 &   & 0 & \\
 1 &  & 20 &  & 1\\
  & 0 &   & 0 & \\
   &   & 1 &&
\end{array}
\]
\end{example}

\begin{example}
A Calabi--Yau variety of dimension $d=3$  is an \emph{Calabi--Yau threefold}, from now on abbreviated into \emph{CY 3-fold}.
It has the following Hodge diamond.
\[
\begin{array}{ccccccc}
    & & &1& & &  \\
    & &0& &0& &  \\
    &0& & h^{1,1}(Y) & &0&  \\
   1& & h^{2,1}(Y) & & h^{1,2}(Y) & &1  \\
    &0& & h^{2,2}(Y) & &0&  \\
    & &0& &0& &  \\
    & & &1& & &  \\
\end{array}
\]
\end{example}

\begin{remark}
CY manifolds are K\"{a}hler, hence $h^{1,1}(Y)>0$.
For CY 3-folds, all $2$-cycles are algebraic, hence $H^2(Y,\IZ)\cong \Pic Y$, that is, $h^{1,1}(Y)=h^2(Y)=\rho (Y)$, the \emph{Picard number} of $Y$.
\end{remark}
\begin{definition}
We say that a CY 3-fold $Y$ is \emph{rigid} if $h^{2,1}(Y)=0$,
otherwise we call it $\emph{non-rigid}$.
\end{definition}

As $Y$ is defined over $\IQ$, it admits a model over $\IZ$.
\begin{definition}
Let $Y$ be a CY variety. 
We say that a prime $p$ is \emph{of good reduction} for $Y$ if 
the reduction $\bar{Y}_p$ of $Y$ over $\overline{\IF}_p$ is again a CY variety.
So, in particular, $\bar{Y}_p$ is smooth.
\end{definition}
Let $p$ be a prime of good reduction for $Y$.
Then the Frobenius map $\Frob_p\colon \overline{\IF}_p\to \overline{\IF}_p, \; x\mapsto x^p$, acts on $\bar{Y}_p$;
it induces an action on the cohomology groups, in particular on $\Het{d}(\bar{Y}_p,\IQ_\ell)\cong \Het{d}(Y,\IQ_\ell)$,
where $\ell\neq p$ is a prime.
We can then consider the trace of the pullback of $\Frob_p$ on $\Het{d}(Y,\IQ_\ell)$,
denoted by
\[
\trace (\Frob_p^* \, |\,\Het{d}(Y,\IQ_\ell)).
\]
\begin{definition}
We recall that the \emph{$L$-function} of $Y$ (or $H^d(Y))$ is a series
\[
L(Y,s):=\sum_{k=1}^\infty \frac{a_k(Y)}{k^s},
\]
and one has
\begin{itemize}
    \item $a_1(Y)=1$,
    \item for every prime $p$ of good reduction, $a_p= \trace (\Frob_p^* \, |\,\Het{d}(Y,\IQ_\ell))$.
\end{itemize} 
\end{definition}
\begin{theorem}[Lefschetz trace formula]
Let $Y$ be a smooth variety over $\IF_p$ and, for $r\geq 1$, let $\#Y_{p^r}$ denote the number of points of $Y$ over $\IF_{p^r}$.
Then
$$
\#Y_{p^r} = \sum_{i=0}^{2d} (-1)^i\trace ((\Frob_p^*)^r \, |\,\Het{d}(\bar{Y},\IQ_\ell)).
$$
\end{theorem}
\begin{theorem}[Weil--Deligne]
Let $Y$ be a smooth complex variety that can be defined over $\IQ$ and let $p$ be a prime of good reduction for $Y$.
Then, for every $i=0,...,2d$ the quantity $\trace (\Frob_p^* \, |\,\Het{i}(Y,\IQ_\ell))$ is an integer and 
\[
| \trace (\Frob_p^* \, |\,\Het{i}(Y,\IQ_\ell)) |\leq h^i(Y)p^{i/2}.
\]
So, in particular,
\[
|a_p(Y)|\leq h^d(Y)p^{d/2}.
\]
\end{theorem}
Using the knowledge of the Hodge diamond it is possible to give formulas to compute $a_p(Y)$ when $Y$ is a CY variety that can be defined over $\IQ$, see below.
\begin{example}
Let $E$ be an elliptic curve, $p$ a prime of good reduction and $\ell\neq p$ another prime,
then 
\begin{itemize}
    \item $a_1(E)=1$,
    \item $a_p(E) = \trace(\Frob^*_p \,|\, \Het{1}(E,\IQ_\ell))$,
    \item $\# E_p=\sum_{i=0}^{2}(-1)^i\trace(\Frob^*_p \,|\, \Het{i}(E,\IQ_\ell))=1-a_p(E)+p$,
\end{itemize}
from which it follows that $a_p(E)=p+1-\#E_p$.
\end{example}
\begin{example}\label{e:K3}
Let $Y$ be a complex K3 surface that can be defined over $\IQ$, $p$ a prime of good reduction and $\ell\neq p$ another prime,
then 
\begin{itemize}
    \item $a_1(Y)=1$,
    \item $a_p(Y) = \trace(\Frob^*_p \,|\, \Het{2}(Y,\IQ_\ell))$,
    \item $\# Y_p=\sum_{i=0}^{4}(-1)^i\trace(\Frob^*_p \,|\, \Het{i}(Y,\IQ_\ell))=1-0+a_p(Y)-0+p^2$,
\end{itemize}
from which it follows that $a_p(Y)=\#Y_p-1-p^2$.
Notice that the same computations hold for any smooth surface with $h^1=h^3=0$.
\end{example}
\begin{theorem}[Modularity of singular K3 surfaces \cite{Liv95}]\label{t:ModularitySingK3}
Let $S$ be a K3 surface defined over $\IQ$ having maximal Picard number over $\overline{Q}$,
and let $q=e^{2\pi i z}$ with $\im (z) >0$.
Then there exists a newform 
$$
f(q)=\sum_{k=1}^\infty b_kq^k \in S^{\textrm{new}}_3(\Gamma (N))
$$ 
of weight $3$ with respect to a congruence subgroup $\Gamma\leq \PSL_2(\IZ)$ and $N=-m^2\cdot \det\Pic S$ for some $m\in\IN$ (cf.~\cite[Prop. 7.1]{Sch05}),  such that
$$
a_p(Y) = b_p
$$
for all the primes $p$ of good reduction for $Y$.
\end{theorem}

\begin{example}\label{e:CY3fold}
Let $Y$ be a CY 3-fold that can be defined over $\IQ$, 
$p$ a prime of good reduction and $\ell\neq p$ another prime,
then 
\begin{itemize}
    \item $a_1(Y)=1$,
    \item $a_p(Y) = \trace(\Frob^*_p \,|\, \Het{3}(Y,\IQ_\ell))$,
    \Item 
    \begin{align*}
        \# Y_p&=\sum_{i=0}^{6}(-1)^i\trace(\Frob^*_p \,|\, \Het{i}(Y,\IQ_\ell))=\\
        &=1-0+ \trace(\Frob_p^*\,|\, H^2(Y,\IQ_\ell))-a_p(Y)
    + \trace(\Frob_p^*\,|\, H^4(Y,\IQ_\ell))-0+p^3.
    \end{align*}
\end{itemize}
By definition of CY 3-fold, we have that $H^2(Y,\IC)=H^{1,1}(Y)$ and hence $\Pic Y= H^2(Y,\IZ)$. 
Then 
\[
\trace(\Frob_p^*\,|\, H^2(Y,\IQ_\ell)) = k_p(Y)p,
\]
where $k_p(Y)\in\IZ$ and $|k_p(Y)|\leq h^2(Y)=h^{1,1}(Y)$.
The equality $k_p(Y)=h^{1,1}(Y)$ holds if and only if $\Het{2}(Y,\IQ_\ell)=\Pic (Y_{\overline{Q}}) \otimes_\IZ \IQ_\ell$ is generated by divisors defined over $\IQ$.
By Poincar\'e duality it follows that 
\[
\trace(\Frob_p^*\,|\, H^4(Y,\IQ_\ell)) = k_p(Y)p^2,
\]
yielding
\[
a_p(Y)=1+(p+p^2)k_p(Y)+p^3-\#Y_p.
\]

\end{example}

\begin{theorem}[Modularity of rigid CY 3-folds \cite{GY11}]\label{t:Modularity}
Let $Y$ be a rigid CY 3-fold, and $q=e^{2\pi i z}$, with $\im (z)>0$.
Then, for some integer $N$, there exists a new form 
\[
f(q)=\sum_{k=1}^\infty b_kq^k\in S_4^{\textrm{new}}(\Gamma_0(N))
\] 
of weight $4$ with respect to $\Gamma_0 (N)$, such that 
\[
b_p = a_p(Y)
\]
for all the primes $p$ of good reduction for $Y$.
The level $N$ is only divisible by the primes of bad reduction of $Y$;
the maximal exponent $e_p$ of a prime dividing $N$ is bounded:
\begin{enumerate}
    \item $e_p\leq 2$ for $p>3$;
    \item $e_3\leq 5$;
    \item $e_2\leq 8$.
\end{enumerate}
\end{theorem}
\begin{remark}\label{r:GeneralisedModularity}
A similar result should hold for every smooth complex variety $Y$ of dimension $2k-1$ with $h^{2k-1}(Y,\IC)=2$ and $h^{2k-1,0}(Y)=1$ if we replace $S_4^{\textrm{new}}$ with $S_{2k}$, see~\cite[\S 2]{GY11}.
So in particular it should hold for possibly singular threefolds $Z$ such that its desingularization $\Tilde{Z}$ has $h^{3}(\Tilde{Z},\IC)=2$ and $h^{3,0}(\Tilde{Z})=1$.
\end{remark}

\subsection{Non-unirationality via point counting}
Let $Y$ be a complex variety, possibly non-smooth, 
and let $\Tilde{Y}\to Y$ be a smooth resolution of singularities.

\begin{assumption}\label{a:ProjRes}
The resolution $\Tilde{Y}\to Y$ is defined over $\IQ$ and $\Tilde{Y}$ is projective.
\end{assumption}

\begin{definition}\label{d:GoodReduction}
By assumption, $Y$ has a non-singular projective model $\Tilde{Y}$  that can be defined over $\IQ$ and hence over $\IZ$.
This means that there exists a $\IZ_p$-scheme $\cY\to \Spec \IZ_p$ such that its generic fiber $\cY_\eta$ is isomorphic to $\Tilde{Y}$.

We say that a prime $p$ is of good reduction for $Y$  if
$\cY_p/\IF_q$ is a smooth variety, for any $q=p^m$.

Notice that $\cY_p$ can be viewed as the reduction $\Tilde{Y}_p$ of $\Tilde{Y}$ modulo $p$.
\end{definition}

\begin{assumption}\label{a:Congruence}
For every prime of good reduction $p$,
 \[
 \#\cY_p(\IF_p)\equiv \# Y_p(\IF_p) \mod p\, .
 \]
\end{assumption}

\begin{lemma}\label{l:Esnault}\cite[Theorem 1.1]{BER12}
Let $\cY\to \Spec \IZ_p$ be a smooth $\IZ_p$-scheme of dimension $d$.
Let $\cY_\eta$ denote the fiber above  $\eta=(0)\in\Spec \IZ_p$, and $\cY_p$ the closed fiber above $(p)$.
If $h^{k,0}(\cY_\eta)=0$ for every $k=1,...,d$,
then $\#\cY_p (\IF_p)\equiv 1 \mod p$.
\end{lemma}

Using~\autoref{l:Esnault}, one can make a first guess about the unirationality of $Y$ via point counting under~\autoref{a:Congruence},
as shown by the following result.

\begin{corollary}\label{c:NotUnirat}
Let $Y$ be defined as before and
let $p$ be a prime of good reduction for $Y$.
Under~\autoref{a:Congruence},
if $\# Y_p(\IF_p) \not\equiv 1 \mod p$, 
then $Y$ is not unirational.
\end{corollary}
\begin{proof}
Assume that $Y$ is unirational, then it follows that $\Tilde{Y}$ and hence $\cY_\eta$ are unirational too.
In particular, $h^{k,0}(\cY_\eta)=0$ for $k=1,2,3$,
Then, as $p$ is a prime of good reduction for $Y$,
c.f.~\autoref{d:GoodReduction}, 
we have that $\#\cY_p (\IF_p) \equiv 1 \mod p$ by~\autoref{l:Esnault}.
Moreover,  by~\autoref{a:Congruence}, it holds that
$\#\cY_p(\IF_p) \equiv \# Y_p (\IF_p) \mod p$.
So, in particular, 
\[
1 \equiv \#\cY_p (\IF_p) \equiv \# Y_p (\IF_p) \not\equiv 1 \mod p \; ,
\]
a contradiction.
This proves that  $Y$ is not unirational.
\end{proof}

We can apply the above tools to the more specific cases of K3 surfaces and Calabi--Yau threefolds,
for which, in some cases, we can use the modularity theorems to perform a consistency check.
In what follows, we will need the following definition.
\begin{definition}\label{d:Sigma}
Let $B\in \IN$ a positive integer. Then we define the sets:
\begin{align*}
    \Sigma(B)&=\{ p \textrm{ prime } : \; p\leq B \textrm{ and } p \textrm{ is of good reduction for } Y \}\, ,\; and\\
    \Sigma_0(B)&=\{ p\in\Sigma (B) \; : \; \# Y_p(\IF_p) \not\equiv 1 \mod p \}.
\end{align*}
\end{definition}


\subsection{Point counting for K3 surfaces}

In this subsection we assume that $\dim Y=2$ and $Y$~satisfies assumptions \ref{a:ProjRes} and \ref{a:Congruence}.
We denote by $\Tilde{Y}$ a fixed minimal smooth model of $Y$.

\begin{lemma}\label{l:CongrK3}
If $\Tilde{Y}$ is a K3 surface,
then
the following congruence holds:
\[
a_p(\Tilde{Y})\equiv \# Y_p(\IF_p)-1 \mod p
\]
for every prime of good reduction $p$.
\end{lemma}
\begin{proof}
As $\Tilde{Y}$ is a K3 surface and $p$ is a prime of good reduction, we know that
\[
a_p(\Tilde{Y})=\#\cY_p(\IF_p)-1-p^2,
\]
see~\autoref{e:K3}.
By~\autoref{a:Congruence},
this yields
\[
a_p(\Tilde{Y})\equiv\#Y_p(\IF_p)-1\mod p,
\]
concluding the proof.
\end{proof}

If we suspect that 
$Y$ is not unirational (for example by \autoref{c:NotUnirat}) 
and, moreover, 
that its desingularization $\Tilde{Y}$ is a singular K3 surface, i.e.\ a K3 surface with maximal Picard number over $\overline{\IQ}$,
we can use \autoref{l:CongrK3} to do a consistency check via point counting, instead of explicitly computing the cohomology of $\Tilde{Y}$.
To do so, we consider the set of primes of good reduction $\Sigma (B)$, cf.~\autoref{d:Sigma},  and we compute the number of points $\# Y(\IF_p)$ for $p\in\Sigma (B)$.
In the examples below, this set does not need to be very large, 
we took $B=100$. 
So, upon computing the number of points $\# Y_p(\IF_p)$ for a finite set $\Sigma (B)$ of primes of good reduction,
we  can look for levels $N$, congruence subgroups $\Gamma\leq \PSL_2 (\IZ)$ and modular forms
\[
f(q)=\sum_{k=1}^\infty b_kq^k\in S_3^{\textrm{new}}(\Gamma(N))
\]
such that $b_p\equiv \# Y_p(\IF_p)-1 \mod p$ for any $p\in\Sigma (B)$.

\begin{assumption}\label{a:ModularCongrK3}
For some $N\in\IN$, subgroup $\Gamma\leq \PSL_2(\IZ)$, and a positive integer $B$ such that $\#\Sigma_0 (B)\geq 10$, 
there exists a newform $f\in S_3^{\textrm{new}}(\Gamma (N))$ such that
\[
b_p\equiv \# Y_p(\IF_p)-1   \mod p
\]
for every $p\in\Sigma (B)$.
\end{assumption}

\begin{conjecture}\label{c:SingK3}
If $Y$ satisfies~\autoref{a:ModularCongrK3},
then $Y$ is birationally equivalent to a smooth surface $\Tilde{Y}$ with $h^{2,0}(\Tilde{Y})\geq 1$.
In particular, $Y$ is not unirational.
\end{conjecture}

\begin{remark}
It is important to take \emph{all} the primes up to a certain bound $B$:
given a surface $Y$ such that $\Tilde{Y}$ is a singular K3 surface,
it might still happen that all for the primes $p$ satisfying a certain congruence one has $\# Y_p(\IF_p)\equiv 1 \mod p$,
giving the false impression that $Y$ is unirational.
In fact this happens for the surface  $S$ in \autoref{ss:ApplicationK3}:
in this case we find that $\# Y_p(\IF_p)\equiv 1 \mod p$ for all the primes $p\equiv 5, 7\mod 8$.
The same happens for the Fermat quartic and all the primes $p\equiv 3\mod 4$.
\end{remark}

\begin{remark}
If $\Tilde{Y}$ is a K3 surface with Picard number $20$ over $\overline{\IQ}$, 
then the level $N$ of the form $f$ determines the discriminant of $\Pic \Tilde{Y}$ modulo squares, cf.~\autoref{t:ModularitySingK3}.
In this case we expect that $b_p\equiv \# Y_p(\IF_p) -1 \mod p$ for all primes $p$ of good reduction.

If $\Tilde{Y}$ is a K3 surface with Picard number $20$ over ${\IQ}$, 
then we expect even more:
$\#\cY_p(\IF_p)=1+20p+b_p+p^2$.
\end{remark}

\subsection{Point counting for Calabi--Yau threefolds}

In this subsection we assume that $\dim Y=3$ and $Y$ satisfies assumptions \ref{a:ProjRes} and \ref{a:Congruence}.
We denote by $\Tilde{Y}$ a fixed smooth model of $Y$.

\begin{remark}
If $Y$ has only ordinary double points, then
\autoref{a:Congruence} is satisfied.
Indeed, an ordinary double point gets resolved after one blow up, introducing a $\IP^1\times \IP^1$ as exceptional divisor.
Then the claim follows from the computations in~\cite[\S 1.6.1]{Mey05}.
\end{remark}

\begin{lemma}\label{l:CongrCY}
If $\Tilde{Y}$ is a Calabi--Yau threefold,
then
the following congruence holds:
\[
a_p(\Tilde{Y})\equiv 1-\# Y_p(\IF_p) \mod p
\]
for every prime of good reduction $p$.
\end{lemma}
\begin{proof}
As $\Tilde{Y}$ is a CY 3-fold, we know that
\[
a_p(\Tilde{Y})=1+(p+p^2)k_p(\Tilde{Y})+p^3-\#\cY_p(\IF_p),
\]
cf.~\autoref{e:CY3fold}.
By~\autoref{a:Congruence},
this yields
\begin{align*}
    a_p(\Tilde{Y})&\equiv 1+(k_p(\Tilde{Y})-\alpha_p)p+(k_p(\Tilde{Y})-\beta_p)p^2+p^3-\# Y_p(\IF_p) \mod p\\
                  &\equiv 1 -\# Y_p(\IF_p) \mod p\, .
\end{align*}
\end{proof}

Analogously to the previous section,
\autoref{l:CongrCY} tells us that if we suspect that $Y$ is not unirational (for example by \autoref{c:NotUnirat}) and that, moreover, its desingularization $\Tilde{Y}$ is a rigid CY 3-fold, 
then we can do a consistency check via point counting instead of explicitly computing the cohomology of $\Tilde{Y}$. 
To do so, it is enough to consider the primes in the set $\Sigma (B)$, cf.~\autoref{d:Sigma}, for some bound $B$.
Also in the example below, we take $B=100$.

We  can then look for levels $N$ and modular forms
\[
f(q)=\sum_{k=1}^\infty b_kq^k\in S_4^{\textrm{new}}(\Gamma_0(N))
\]
such that $b_p\equiv 1-\# Y_p(\IF_p) \mod p$ for any $p\in\Sigma (B)$.
Note that in this case
the congruence subgroup is fixed, 
and 
\autoref{t:Modularity} gives us a hint about the possible values for $N$,
making the search more practical.

\begin{assumption}\label{a:ModularCongr}
For some $N\in\IN$ and a positive integer $B$ such that $\#\Sigma_0 (B)\geq 10$,
there exists a newform $f\in S_4^{\textrm{new}}(\Gamma_0(N))$ such that
\[
b_p\equiv 1-\# Y_p(\IF_p) \mod p
\]
for every $p\in\Sigma (B)$.
\end{assumption}
\begin{conjecture}\label{c:Bert}
If $Y$ satisfies~\autoref{a:ModularCongr},
then $Y$ is birationally equivalent to a smooth threefold $\Tilde{Y}$ with $h^{3,0}(\Tilde{Y})\geq 1$.
In particular, $Y$ is not unirational.
\end{conjecture}

\subsection{Application to two surfaces}\label{ss:ApplicationK3}
Let $Q\subset \IP^3$ be the quartic surface defined by
\begin{equation*}
    Q\colon \; -x_3^2(x_4^2-x_1x_2)+(x_1+x_2)(x_1x_4^2+x_2x_4^2-4x_1x_2x_4+x_1^2x_2+x_1x_2^2)=0,
\end{equation*}
and  $S\subset \IP(3,1,1,1)$ the double sextic defined by
\begin{equation*}
    S\colon \; y_0^2=(y_1^2+y_2^2-2y_3^2)(y_1^2y_2^2-y_3^4),
\end{equation*}
both studied in~\cite{FvS19}.
The two surfaces have several singular points, so we cannot immediately conclude that they are birationally equivalent to K3 surfaces.
After counting points of the reductions modulo all the odd primes smaller than $100$, 
one sees that $\# Q_p (\IF_p), \# S_p (\IF_p)\not\equiv 1 \mod p$ for many primes $p$.
Moreover $Q$ and $S$ both satisfy~\autoref{a:ModularCongrK3} with
\begin{align*}
    f(q)&=q - 6q^5 + 9q^9 + 10q^{13} -30q^{17}+11q^{25}+42q^{29}-70q^{37}+ O(q^{41})\in S_3^{\textrm{new}}(\Gamma_1(16)),\\
    g(q)&=q - 2q^2 - 2q^3 + 4q^4 + 4q^6 - 8q^8 - 5q^9 + 14q^{11} -8q^{12} +O(q^{14})\in S_3^{\textrm{new}}(\Gamma_1(8))
\end{align*}
respectively.
Therefore we expect $Q$ and $S$ not to be unirational.

The prediction for $S$ is confirmed by~\cite[Theorem]{FvS19},
where it is shown that its smooth model $\Tilde{S}$  a K3 surface  whose Picard lattice has rank $20$ and discriminant $-8$.
In the same paper, the authors suggest that $Q$ and $S$ are birationally equivalent, which contradicts the heuristic associating them to two distinct modular forms, moreover of levels differing by $2$, which is not a square.
Below we show that the suggestion is indeed wrong.

\begin{proposition}
 The surfaces $Q$ and $S$ are not birationally equivalent.
\end{proposition}
\begin{proof}
Let $\Tilde{Q}$ and $\Tilde{S}$ denote the smooth models of $Q$ and $S$, respectively.
It is easy to check that $Q$ has only ADE singularities, and hence $\Tilde{Q}$ is a K3 surface.
The same holds for $S$ and $\Tilde{S}$.
Recall that if two K3 surfaces are birationally equivalent then they are isomorphic.
This means that if $Q$ and $S$ are birationally equivalent, $\Tilde{Q}$ and $\Tilde{S}$ are isomorphic;
in particular, they have isometric Picard lattices.

From~\cite[Theorem]{FvS19} we know that $\Pic \Tilde{S}$ has rank $20$ and  discriminant $-8$.
By projecting from a node of $Q$, we obtain a model of $Q$ as double sextic.
It is then easy to find some $-2$-curves on this model.
Using these curves and the exceptional divisors we generate a lattice $\Lambda\subseteq\Pic \Tilde{Q}$ of rank $20$, discriminant $-16$ and index $[\Pic \Tilde{Q} : \Lambda]\leq 2$.
It follows that $\Pic \Tilde{Q}$ has rank $20$ and discriminant either $-16$ or $-4$.
In any case, $\det \Pic \Tilde{Q}\not\cong \det \Pic \Tilde{S}$,
showing that $Q$ and $S$ are not birationally equivalent.
\end{proof}
\begin{remark}
Note that the level of $f$ is a square multiple of $\det \Pic \Tilde{Q}$, as predicted by~\autoref{t:ModularitySingK3}.
The same holds for $g$ and $S$.
\end{remark}

\subsection{Application to the double octic}\label{ss:GuessingWeinzierlOctic}
Let $X$ be the double octic defined in~\eqref{eq:DoubleOctic}.
Notice that  it is reasonable to assume that $X$ satisfies \ref{a:ProjRes} and~\ref{a:Congruence}.
We proceed then with the point counting, obtaining~\autoref{t:PointsXp}.

\begin{table}[ht]
\begin{tabular}{|l|llllllll|}
\hline
$p$ & 3 & 5 & 7 & 11 & 13 & 17 & 19 & 23  \\
\hline
$\# X_p (\IF_p)$ & 46 & 180 & 500 &  1716 & 2732 & 6060 & 8132 & 13932 \\
\hline
$1-\# X_p (\IF_p)\mod p$ & 0 & 1 & 5 & 1 & 12 & 10 & 1 & 7 \\
\hline
\hline
$p$ & 29 & 31 & 37 & 41 & 43 & 47 & 53 & 59  \\
\hline
$\# X_p (\IF_p)$ & 27492 & 33476 & 55580 &  75276 & 86612 & 112380 & 159492 & 219492\\
\hline
$1-\# X_p (\IF_p)\mod p$ & 1 & 5 & 32 & 1 & 34 & 45 & 39 & 48 \\
\hline
\hline
$p$ & 61 & 67 & 71 & 73 & 79 & 83 & 89 & 97  \\
\hline
$\# X_p (\IF_p)$ & 241916 & 317300 & 376716 &  409532 & 517892 & 599172 & 735132 &  948380\\
\hline
$1-\# X_p (\IF_p)\mod p$ & 11 & 13 & 11 & 72 & 33 & 6 & 9 & 87 \\
\hline
\end{tabular}
\caption{Point counting on $X_p$ for every odd prime smaller than 100.}
\label{t:PointsXp}
\end{table}
For many primes $p$ we see that $X_p (\IF_p)\not\equiv 1\mod p$, hence 
by~\autoref{c:NotUnirat} we expect $X$ not to be unirational.
In an attempt to gather more evidence towards the non-unirationality of $X$, we look for a modular form satisfying~\autoref{a:ModularCongr}.
We then find that the unique newform $f(q)=\sum_{k=1}^\infty b_kq^k\in S_4^{\textrm{new}}(\Gamma_0(6))$, with
\begin{equation}\label{eq:NF604}
    f(q) = q - 2q^2 - 3q^3 + 4q^4 + 6q^5 + 6q^6 - 16q^7 - 8q^8 + 9q^9 - 12q^{10} + 12q^{11} + O(q^{12}),
\end{equation}
satisfies~\autoref{a:ModularCongr}, giving yet a stronger hint, cf.~\autoref{c:Bert}.
The heuristics suggesting that $X$ is not unirational are indeed confirmed by~\autoref{p:XCY}.
\begin{remark}\label{r:Equality}
In this particular case, the evidence is even stronger:
indeed, one can check that
\begin{equation}\label{eq:bp}
    b_p=1-8p +4 p^2 + p^3 - \# X_p (\IF_p)
\end{equation}
for every prime $3 < p < 100$,
suggesting that $\Tilde{X}$ is a \emph{rigid} CY 3-fold.
This is confirmed by B. Naskr\c{e}cki in a private communication:
he proves that the $L$-function of $\Tilde{X}$ is indeed \eqref{eq:NF604} and deduces the rigidity of $\Tilde{X}$.
The same result was suggested also by S. Cynk.

In~\cite{vdBE05}, van den Bogaart and Edixhoven compute the cohomology of varieties defined over $\IQ$ and satisfying even stronger conditions on the number of rational points of their reduction modulo a prime $p$.
In particular, they show that if $Y/\IQ$ is a smooth variety such that $\# Y_p(\IF_p)$ behaves polynomially for all primes $p$, then all the cohomology groups of odd degree vanish and they determine the dimension of the cohomology groups of even degree.
\end{remark}

\begin{remark}
The fact that \eqref{eq:bp} does not hold for $p=3$
is consistent with  the level of $f$ being $6$, 
suggesting that $p=3$ is a prime of bad reduction for $X$, cf.~\autoref{t:Modularity}. 
Indeed, if $X_3$ denotes the reduction of $X$ modulo $3$,
then, looking at~\autoref{t:SingularPoints}, one can see that  the four singular points 
$$
(1:1:1:0),(1:1:4:3), (1:4:1:3), (4:1:1:3)\in X
$$ 
collapse to the singular point $(1:1:1:0)\in X_3$.
One can then check that $X_3$ has a non-ordinary singularity in $(1:1:1:0)$,
whereas $X$ has an ordinary one.
\end{remark}
\begin{remark}
The same newform $f$ in~\eqref{eq:NF604} appears also in~\cite{HSvGvS01} and~\cite[\S 5.7]{Mey05}, in relation to other CY 3-folds.
\end{remark}

\bibliographystyle{amsplain}
\bibliography{references}

\providecommand{\bysame}{\leavevmode\hbox to3em{\hrulefill}\thinspace}
\providecommand{\MR}{\relax\ifhmode\unskip\space\fi MR }
\providecommand{\MRhref}[2]{%
  \href{http://www.ams.org/mathscinet-getitem?mr=#1}{#2}
}
\providecommand{\href}[2]{#2}
\begin{thebibliography}{10}

\bibitem{BER12}
Pierre Berthelot, H\'{e}l\`ene Esnault, and Kay R\"{u}lling, \emph{Rational
  points over finite fields for regular models of algebraic varieties of
  {H}odge type {$\geq1$}}, Ann. of Math. (2) \textbf{176} (2012), no.~1,
  413--508. \MR{2925388}

\bibitem{BF21}
Marco Besier and Dino Festi, \emph{Rationalizability of square roots}, J.
  Symbolic Comput. \textbf{106} (2021), 48--67. \MR{4199812}

\bibitem{Cyn02}
S{\l}awomir Cynk, \emph{Cohomologies of a double covering of a non-singular
  algebraic 3-fold}, Math. Z. \textbf{240} (2002), no.~4, 731--743.
  \MR{1922727}

\bibitem{CS98}
S{\l}awomir Cynk and Tomasz Szemberg, \emph{Double covers and {C}alabi-{Y}au
  varieties}, Singularities {S}ymposium---\L ojasiewicz 70 ({K}rak\'{o}w, 1996;
  {W}arsaw, 1996), Banach Center Publ., vol.~44, Polish Acad. Sci. Inst. Math.,
  Warsaw, 1998, pp.~93--101. \MR{1677335}

\bibitem{CvS06}
S{\l}awomir Cynk and Duco van Straten, \emph{Infinitesimal deformations of
  double covers of smooth algebraic varieties}, Math. Nachr. \textbf{279}
  (2006), no.~7, 716--726. \MR{2226407}

\bibitem{FH22}
Dino Festi and Andreas Hochenegger, \emph{Rationalizability of field extensions
  with a view towards {F}eynman integrals}, J. Geom. Phys. \textbf{178} (2022),
  Paper No. 104536. \MR{4420531}

\bibitem{FvS19}
Dino Festi and Duco van Straten, \emph{Bhabha scattering and a special pencil
  of {K}3 surfaces}, Commun. Number Theory Phys. \textbf{13} (2019), no.~2,
  463--485. \MR{3951114}

\bibitem{GY11}
Fernando~Q. Gouv\^{e}a and Noriko Yui, \emph{Rigid {C}alabi-{Y}au threefolds
  over {$\mathbb{Q}$} are modular}, Expo. Math. \textbf{29} (2011), no.~1,
  142--149. \MR{2785550}

\bibitem{HSvGvS01}
Klaus Hulek, Jeroen Spandaw, Bert van Geemen, and Duco van Straten, \emph{The
  modularity of the {B}arth-{N}ieto quintic and its relatives}, Adv. Geom.
  \textbf{1} (2001), no.~3, 263--289. \MR{1874236}

\bibitem{KRW21}
Philipp~Alexander Kreer, Robert Runkel, and Stefan Weinzierl, \emph{{Feynman
  integrals for binary systems of black holes}}, {15th International Symposium
  on Radiative Corrections: Applications of Quantum Field Theory to
  Phenomenology and LoopFest XIX: Workshop on Radiative Corrections for the LHC
  and Future Colliders}, 10 2021.

\bibitem{Liv95}
Ron Livn\'{e}, \emph{Motivic orthogonal two-dimensional representations of
  {${\rm Gal}(\overline {\mathbb{Q}}/ {\mathbb{Q}})$}}, Israel J. Math.
  \textbf{92} (1995), no.~1-3, 149--156. \MR{1357749}

\bibitem{Mey05}
Christian Meyer, \emph{Modular {C}alabi-{Y}au threefolds}, Fields Institute
  Monographs, vol.~22, American Mathematical Society, Providence, RI, 2005.
  \MR{2176545}

\bibitem{Sch05}
Matthias Sch{\"u}tt, \emph{{CM} newforms with rational coefficients},
  \href{https://arxiv.org/abs/math/0511228}{arXiv:0511228v5}, 2005.

\bibitem{vdBE05}
Theo van~den Bogaart and Bas Edixhoven, \emph{Algebraic stacks whose number of
  points over finite fields is a polynomial}, Number fields and function
  fields---two parallel worlds, Progr. Math., vol. 239, Birkh\"{a}user Boston,
  Boston, MA, 2005, pp.~39--49. \MR{2176584}

\bibitem{Voi07}
Claire Voisin, \emph{Hodge theory and complex algebraic geometry. {I}}, english
  ed., Cambridge Studies in Advanced Mathematics, vol.~76, Cambridge University
  Press, Cambridge, 2007, Translated from the French by Leila Schneps.
  \MR{2451566}

\bibitem{Wei22}
Stefan Weinzierl, \emph{Feynman integrals: A comprehensive treatment for
  students and researchers}, UNITEXT for Physics, Springer International
  Publishing, 2022.

\end{thebibliography}

\end{document}